%% file: main.tex
\documentclass{article}
\usepackage{amsmath, amssymb, amsthm, color, verbatim,enumerate}
\usepackage{graphicx}
\usepackage{hyperref}
\usepackage{tikz}

\newtheorem{prevtheorem}{Theorem}

\newtheorem{question}{Question}
\newtheorem{theorem}{Theorem}[section]
\newtheorem{lemma}[theorem]{Lemma}
\newtheorem{corollary}[theorem]{Corollary}

\DeclareMathOperator{\CD}{\mathcal{CD}}

\title{On groups with few subgroups not in the Chermak--Delgado lattice}
\author{David Burrell\footnote{David Burrell, University of Florida, davidburrell@ufl.edu},\,  William Cocke\footnote{William Cocke, Augusta University School of Computer and Cyber Sciences, wcocke@augusta.edu},\, Ryan McCulloch\footnote{Ryan McCulloch, Elmira College, rmcculloch@elmira.edu}}
\date{Nov 2022}

\begin{document}

\maketitle

\begin{abstract}
We investigate the question of how many subgroups of a finite group are not in its Chermak--Delgado lattice. The Chermak--Delgado lattice for a finite group is a self-dual lattice of subgroups with many intriguing properties. Fasol{\u{a}} and T{\u{a}}rn{\u{a}}uceanu \cite{FT} asked how many subgroups are not in the Chermak--Delgado lattice and classified all groups with two or less subgroups not in the Chermak--Delgado lattice. We extend their work by classifying all groups with less than five subgroups not in the Chermak--Delgado lattice. In addition, we show that a group with less than five subgroups not in the Chermak--Delgado lattice is nilpotent. In this vein we also show that the only non-nilpotent group with five or fewer subgroups in the Chermak--Delgado lattice is $S_3$.
\end{abstract}

\section{Introduction.}
In this paper we examine the question of how many subgroups of a finite group $G$ are not   in the Chermak--Delgado lattice of $G$. Chermak and Delgado \cite{CD89} first defined the Chermak--Delgado lattice in 1989. In 2022, Fasol{\u{a}} and T{\u{a}}rn{\u{a}}uceanu \cite{FT} investigated the question of how large the Chermak--Delgado lattice of a finite group can be. They classified all groups with at most $2$ subgroups not in their Chermak--Delgado lattice. In Theorem \ref{thm: lt5}, we extend their work by classifying all groups with less than five subgroups not in their Chermak--Delgado lattice. 

Throughout the paper, we will use the notation $H\leq G$ to mean that $H$ is a subgroup of $G$. The Chermak--Delgado lattice is a sublattice of the subgroup lattice of $G$. To define the Chermak--Delgado lattice of a finite group $G$, which we write as $\CD(G)$, we first need to define a function $m_{G}$ called the Chermak--Delgado measure. The function $m_{G}$ takes as input a subgroup of $G$ and returns the product of the order of the subgroup and the order of its centralizer in $G$, i.e., \[m_{G}(H) = |H| \cdot |\textbf{C}_G(H)|.\] It is surprising that the set of subgroups with maximum Chermak--Delgado measure have a very special property: they form a lattice, the Chermak--Delgado lattice.

For a finite group $G$, we write $m^{*}(G)$ for the maximum value of $m_G$ on a finite group, i.e., \[m^{*}(G) = \max_{H\leq G}\left\{ m_G(H)\right\}.\] If the group is clear from context, we will shorten $m^{*}(G)$ to just $m^*$. If $H,K \leq G$ satisfy $m_G(H)=m_G(K) = m^{*}$,  then $HK = \langle H, K\rangle$ and $m_G(HK) = m_G(H\cap K) = m^{*}.$ Hence the subgroups with maximum Chermak--Delgado measure form a sublattice of the subgroup lattice of $G$: this is the Chermak--Delgado lattice of $G$. The proof that $\CD(G)$ is a lattice can be found in Isaacs \cite[1.G]{FGT}. 

Before stating our main results, we introduce some new notation. We write $\delta_{\CD}(G)$ for the number of subgroups of $G$ not in $\CD(G)$. Hence, $\delta_{\CD}(G) = 0$, if all subgroups of $G$ are in $\CD(G)$ and $\delta_{\CD}(G) = 1$ if there is single subgroup of $G$ not   in $\CD(G)$. The cyclic groups show that $\delta_{\CD}$ maps onto the natural numbers. 

Fasol{\u{a}} and T{\u{a}}rn{\u{a}}uceanu classified all groups where $\delta_{\CD}(G)\leq 2$ \cite[Theorem 1.1]{FT} and asked for such a classification when $\delta_{\CD}(G) > 2$. In this paper, we provide such a classification when $\delta_{\CD}(G) \leq 4$. To do this, we first investigate how $\delta_{CD}(G)$ influences the structure of a finite group $G$. 
\begin{prevtheorem}\label{thm: nil}
Let $G$ be a finite group. If $\delta_{\CD}(G) < 5$, then $G$ is nilpotent.  
\end{prevtheorem}

Of note, when $\delta_{\CD}(G) = 5$ and $G$ is not nilpotent, we have the following theorem where $S_3$ is the symmetric group on $3$ symbols. 

\begin{prevtheorem}\label{thm: s3}
Let $G$ be a finite group that is not nilpotent. If $\delta_{\CD}(G) = 5$, then $G \cong S_3$. 
\end{prevtheorem}

Using Theorem \ref{thm: nil} together with a mixture of computational and theoretical results, we are able to complete the classification of groups with $\delta_{\CD}(G) < 5$. 

\begin{prevtheorem}\label{thm: lt5}
Let $G$ be a finite group. If $\delta_{\CD}(G) <5$, then one of the following holds:

\begin{enumerate}
    \item $\delta_{\CD}(G) < 3$;
    \item $\delta_{\CD}(G) = 3$ and $G$ is cyclic of order either $p\cdot q$ or $p^3$ for primes $p, q$;
    \item $\delta_{\CD}(G) = 4$ and $G$ is isomorphic to either a cyclic group of order $p^4$ for a prime $p$, $C_2\times C_2$, or the extraspecial group of order 27 and exponent 9.
\end{enumerate}
\end{prevtheorem}

Hence, we extend the classification of groups with $\delta_{CD}(G)<3$ to $\delta_{CD}(G) < 5$. 

The rest of the paper proceeds as follows: In Section \ref{sec: properties} we collect some basic properties of the Chermak--Delgado lattice. In Section \ref{sec: nil} we prove Theorem \ref{thm: nil} and in Section \ref{sec: s3} we prove Theorem \ref{thm: s3}. Section \ref{sec: lt5} contains our proof of Theorem \ref{thm: lt5}. Finally in our conclusion, we ask some questions motivated by the study of $\delta_{\CD}$. 

\section{Preliminary material.}\label{sec: properties}

In this section we recall a few well-known properties of the Chermak--Delgado lattice for a finite group $G$. We also present some elementary properties of the Chermak--Delgado lattice for a few families of groups. 

\begin{lemma}\label{lem: contain_center}
Suppose $G$ is a finite group.  If $H \in \CD(G)$, then $\textbf{Z}(G) \leq H$.
\end{lemma}

\begin{proof}
If $H \leq G$, then $\textbf{C}_G(H) = \textbf{C}_G(H\textbf{Z}(G))$, and so if $H \in \CD(G)$, and so $H$ attains the maximum Chermak-Delgado measure in $G$, then $\textbf{Z}(G) \leq H$.
\end{proof}

\begin{corollary}\label{cor: p_gp_1_not_in_p}
Suppose that a finite nontrivial group $G$ is a $p$-group.  Then $1 \notin \CD(G)$. \end{corollary}

\begin{proof}
A nontrivial $p$-group has a nontrivial center, and the result follows by Lemma \ref{lem: contain_center}.
\end{proof}

The following result appears in McCulloch \cite[Corollary 7]{M18}.

\begin{lemma}\cite{M18}\label{lem: no_p_groups}
Suppose $G$ is a finite group and $1 \in \CD(G)$.  Then $\CD(G)$ contains no nontrivial $p$-group for any primes $p$.
\end{lemma}

We will often use the contrapositive to Lemma \ref{lem: no_p_groups}, which we state below as a lemma.

\begin{corollary}\label{cor: 1_not_in_p}
Suppose $G$ is a finite group.  If a  nontrivial $p$-group is in $\CD(G)$, then $1 \notin \CD(G)$. \end{corollary}

Another well-known result about Chermak--Delgado lattices of finite groups is the following by Brewster and Wilcox \cite[Theorem 2.9]{BW12}.

%ADD REFERENCE
\begin{lemma}\cite{BW12}\label{lem:dir_prod}
Let $G$ and $H$ be finite groups. Then $\CD(G\times H) = \CD(G) \times  \CD(H)$.
\end{lemma}

Together with Corollary 2.2 above, this implies that in a finite nilpotent group we can greatly restrict the types of groups that appear in the Chermak--Delgado lattice.

\begin{corollary}\label{cor: no_p_groups_nil}
Let $G$ be a finite nilpotent group. If $p$ divides $|G|$, then $p$ divides the order of $H$ for every $H$ in $\CD(G)$. 
\end{corollary}

\begin{proof}
We can write $G$ as $S \times K$ where $S$ is a Sylow $p$-group of $G$ and $K$ is a $p$-complement of $S$. We know that $\CD(G) = \CD(S) \times \CD(K)$. Since $\CD(S)$ does not contain $1$, we conclude that every element of $\CD(S)$ is a nontrivial $p$-group. Hence $p$ divides any $H$ in $\CD(G)$. 
\end{proof}

Another interesting result about the Chermak--Delgado lattice is that since $H\in \CD(G)$ if and only if $H^g \in \CD(G)$, we have that $HH^g$ is a group for all $g\in G$. Recall that a subgroup $H$ of a group $G$ is called subnormal if 
\[ H = K_0 \lhd\, K_1 \lhd\, \dots \lhd\, K_n = G.\] Foguel showed that groups $H$ that permute with their conjugates, i.e., $HH^g$ is a subgroup, are subnormal \cite{F97}. Hence subgroups in $\CD(G)$ are subnormal in $G$, which is also noted in \cite{BW12,cocke20}.

The following lemma is part of a theme we will see throughout the paper: certain conditions on $G$ and $\CD(G)$ will exclude other groups from being in $\CD(G)$. We say that a set of subgroups $X = \{H_1,\dots,H_k : H_i \leq G\}$ of cardinality $k$ \emph{witnesses} that $\delta_{CD}(G) \geq k$, if for each $H_i\in X$ we have that $H_i \not \in \CD(G)$. 

\begin{lemma} \label{lem: order_p_witnesses}
Let $G$ be a finite group. If $G$ has $k$ subgroups of prime order, then $\delta_{CD}(G) \geq k$.
\end{lemma}

\begin{proof}
If $k=0$, then $G$ is trivial and the result is trivially true.  So suppose $k > 0$.
If $1 \in \CD(G)$, then it follows from Lemma \ref{lem: no_p_groups} that the $k$ subgroups of prime order witness that $\delta_{CD}(G) \geq k$.
    Suppose $1\notin \CD(G)$ and suppose that $H, K \leq G$ are two distinct subgroups of prime order and $H,K \in \CD(G)$. Then $1 = H\cap K \in \CD(G)$, a contradiction.  Hence among the subgroups of prime order, at most one of them can be in $\CD(G)$. Thus the other $k-1$ subgroups of prime order, together with the identity witness that $\delta_{CD}(G) \geq k$. 
\end{proof}

The following lemma concerns generalized quaternion groups. Recall that a group is called a generalized quaternion group if $G\cong \langle a,b | a^{2k}, b^4, a^{k} = b^2, a^b = a^{-1} \rangle.$ We will write the generalized quaternion group of order $m$ as $Q_m$. It is well-known that a finite $p$-group with a single subgroup of order $p$ is either cyclic or generalized quaternion \cite[4.4]{S86}. We also need the following from Fasol{\u{a}} and T{\u{a}}rn{\u{a}}uceanu \cite[The proof of Lemma 2.1]{FT}.

\begin{lemma}[\cite{FT}]\label{lem: quaternion}
If $G$ is a generalized quaternion $2$-group, then \[\left|\CD\left(Q_{2^k}\right)\right| =\begin{cases} 5 & \text{ if $ k = 3$}\\ 1 & \text{ if  $k > 3$.}
\end{cases}\]
\end{lemma}

For the next lemma, we will need to introduce some notation. For a group $G$ and two subgroups $H$ and $K$ of $G$, we write $
[[H:K]]_G$ for the set of subgroups between $H$ and $K$, i.e.,
\[
[[H:K]]_G =  \{ J\leq G: H\leq J \text{ and } J \leq K \}.
\] If the group $G$ is clear from context, then we write $[[H:K]]$ for $[[H:K]]_G$. The following lemma from An \cite[Theorem 3.4]{An_22} and \cite[Theorem 4.4]{An_22_2} shows that when $H\in \CD(G)$, we can actually say something abut $\CD(H)$.

\begin{lemma}\label{lem: interval}
Let $G$ be a finite group and $H\leq G$. If $H\in\CD(G)$ then $\CD(H)$ is exactly the set of subgroups of $H$ containing $\textbf{Z}
(H)$ and in $\CD(G)$, i.e., \[\CD(H)=[[\textbf{Z}(H):H]] \cap \CD(G).\]
\end{lemma}

%The last preliminary lemma we need comes from the classification of minimal abelian $p$-group. 

\section{Proof of Theorem \ref{thm: nil}.}\label{sec: nil}
Our proof of Theorem \ref{thm: nil} depends on a connection between the Sylow subgroups and the Chermak--Delgado lattice. One of the Sylow Theorems states that the total number of Sylow $p$-subgroups in $G$ is 1 if and only if any of the Sylow $p$-subgroups is normal in $G$. The below lemma extends this observation.
\begin{lemma}\label{lem: subnormal-normal}
Let $G$ be a finite group and let $S$ be a Sylow subgroup of $G$. Then $S$ is subnormal in $G$ if and only if $S$ is normal in $G$. 
\end{lemma}

\begin{proof}
A normal subgroup is also subnormal. Suppose now that $S$ is subnormal in $G$, and that $S$ is a $p$-group for a prime $p$. This means there is a chain of subgroups 
\[ S = K_0 \lhd\, K_1 \lhd\, \dots \lhd\, K_n = G.\] We note that $S$ is a Sylow subgroup of all of the $K_i$. Since $S\lhd\, K_1$, we know that $S$ is the unique Sylow $p$-subgroup of $K_1$. Hence $S$ is a characteristic subgroup of $K_1$. Thus $S$ is normal in $K_2$ and by the same argument characteristic in $K_2$. Continuing in this manner, we see that $S$ is normal in $G$. 
\end{proof}

\begin{lemma}\label{lem: at_most_1}
Let $G$ be a finite group. Then at most one Sylow subgroup of $G$ is in $\CD(G)$.  Furthermore, if a Sylow subgroup of $G$ is in $\CD(G)$, then it is normal.
\end{lemma}

\begin{proof}
Since all subgroups in $\CD(G)$ are subnormal, we have by Lemma \ref{lem: subnormal-normal} that if a Sylow subgroup of $G$ is in $\CD(G)$, then it is normal. 

First suppose that $1\in \CD(G)$.  Now $1$ is a Sylow subgroup of $G$ for any prime not dividing $|G|$.  It follows from Lemma \ref{lem: no_p_groups} that no nontrivial Sylow subgroups of $G$ are in $\CD(G)$, and so in this case, $1$ is the unique Sylow subgroup in $\CD(G)$.

Suppose $p$ is a divisor of $|G|$ and suppose a Sylow $p$-subgroup $S$ of $G$ is in $\CD(G)$. Write $|S| = p^k$. Since $S$ is normal in $G$, we have that $S$ is the unique Sylow $p$-subgroup of $G$.  Let $T$ be another Sylow subgroup of $G$, and so $T$ is not a $p$-group. We show that $T$ is not in $\CD(G)$.

By definition, $S\in \CD(G)$ means that $m^{*}(G) = |S| \cdot |\textbf{C}_G(S)|$. Since $1 < \textbf{Z}(S) \leq \textbf{C}_G(S)$, we conclude that $p^{k+1} | m^{*}(G)$. Since $S$ is a Sylow $p$-subgroup of $G$ with $|S| = p^k$, we know that $p^{k+1}$ does not divide $|\textbf{C}_G(T)|$. And so $p^{k+1}$ does not divide $|T|\cdot |\textbf{C}_G(T)|$ and we conclude that $T$ is not in $\CD(G)$.
\end{proof}

\begin{corollary}\label{cor: sylow_count}
Let $G$ be a finite group and write $n$ for the number of nontrivial Sylow subgroups of $G$. Then $\delta_{\CD}(G) \geq n$. 
\end{corollary}

\begin{proof}
If none of the nontrivial Sylow subgroups of $G$ are in $\CD(G)$, then the nontrivial Sylow subgroups of $G$ witness that $\delta_{\CD(G)} \geq n.$

Otherwise, suppose that a nontrivial Sylow subgroup $S$ of $G$ is in $\CD(G)$.  By Lemma \ref{lem: at_most_1} and Corollary \ref{cor: 1_not_in_p}, \[\{T:T \neq S \text{ is a nontrivial Sylow subgroup of $G$}\} \cup \{1\}\] witnesses that $\delta_{\CD(G)} \geq n.$
\end{proof}

%TODO Define nilpotent and relate back to the Sylow theorems. 

Another fact of Sylow theory is that the number of Sylow $p$-subgroup of a group $G$ is congruent to $1$ modulo $p$, and this number divides the index in $G$ of a Sylow $p$-subgroup. Hence if a finite group $G$ is not nilpotent, then it has at least four Sylow subgroups with equality if and only if $G$ is a $\{2,3\}$-group and the Sylow $3$-subgroup is normal. 

%Before proving Theorem\ref{thm: nil}, we state a result of Wielandt.

%\begin{lemma}\label{lem: wielan}
%Suppose $G$ is a finite group with $|G| = p^n m$ with $(p,m)=1$.  For each $0 \leq k \leq n$, the number of subgroups of order $p^k$ is congruent to $1$ modulo $p$.
%\end{lemma}

\begin{proof}[Prove of Theorem \ref{thm: nil}]
The result for $\delta_{\CD}(G) < 4$ follows from Corollary \ref{cor: sylow_count} and the fact that non-nilpotent groups always have at least four Sylow subgroups. Suppose $\delta_{\CD}(G) = 4$.  Let $S$ be the Sylow $3$-subgroup of $G$.

We first argue that $S \in \CD(G)$.  Suppose by way of contradiction that $S\notin \CD(G)$.  Then the four subgroups that are not in $\CD(G)$ are the three Sylow $2$-subgroups of $G$ and $S$. We conclude that $1 \in \CD(G)$.  It follows from Lemma \ref{lem: no_p_groups} that all of the nontrivial $2$-groups and $3$-groups of $G$ are not in $\CD(G)$.  Since $\delta_{\CD}(G) = 4$, $G$ has exactly four total nontrivial $2$-groups or $3$-groups, namely the three Sylow $2$-subgroups of $G$ and $S$.  So all nontrivial Sylow subgroups must have prime order.  Thus, $m^*(G) = m_G(1)=|G|=6$. But $S$ has a Chermak-Delgado measure of $9$, a contradiction.

Therefore $S \in \CD(G)$, and the four subgroups that are not in $\CD(G)$ are the three Sylow $2$-subgroups of $G$ and the identity subgroup. We conclude that $G \in \CD(G)$.

Note that any Sylow $2$-subgroup of $\textbf{C}_G(S)$ is also a Sylow $2$-subgroup of $G$.  This is because $|S| \cdot |\textbf{C}_G(S)| = m^*(G)$ and since $G \in \CD(G)$, $|G|$ divides $m^*(G)$.  Now we have that $G=ST$ is a direct product where $T$ is a Sylow $2$-subgroup of $\textbf{C}_G(S)$, which implies that $G$ is nilpotent, a contradiction.
\end{proof}

We can say more about the non-nilpotent case, as seen in the next section where we prove Theorem \ref{thm: s3}.

\input{s3_alternative}

\input{alternative_proofs}

\section{Conclusion.}
We have extended the results by Fasol{\u{a}} and T{\u{a}}rn{\u{a}}uceanu \cite{FT} by classifying all finite groups with $3\leq \delta_{\CD}(G)\leq 4$. Obvious questions exist about classifying all groups where $\delta_{\CD}(G) = k$ for $k\geq 5$. In particular, for $\delta_{\CD}(G) = 5$ combining Theorem \ref{thm: s3} and Lemma \ref{lem: lt5->p-group} means one only needs to classify nonabelian $2$-groups and $3$-groups with $\delta_{\CD}(G) = 5$. 

Corollary  \ref{cor: sylow_count}  establishes that $\delta_{\CD}(G)$ is greater than the number of distinct prime divisors of a group $G$. Is it the case that for a fixed $k$ we can also bound the multiplicity of a prime divisor of $G$? We ask this in three distinct, but highly related questions:

\begin{question}
Let $n$ be a positive integer not equal to 8. If $G$ has order $n$, then is $\delta_{\CD}(G) \geq \delta_{\CD}(C_n)$ = number of proper divisors of $n$. 
\end{question}

\begin{question}
Let $G$ be a finite group. If $G$ has exactly $k$ prime divisors up to multiplicity, is it the case that $\delta_{\CD}(G)$ is bounded by a function depending on $k$?
\end{question}

\begin{question}

Let $G$ be a finite group. If $H\leq G$ has order $p^k$ for some prime $p$ and integer $k$, then is 
\[\delta_{\CD}(G) \geq \begin{cases} k & \text{ $p$ is odd,} \\k-2 & \text{ $p = 2$.} \end{cases}\]

\end{question}

Theorem \ref{thm: nil} shows that when $\delta_{\CD}(G) < 5$, we have that $G$ is nilpotent. One can ask the same question for solvable groups.

\begin{question}
What is the maximal value of $m$, such that if a finite group satisfies $\delta_{\CD}(G) < m$, then $G$ must be solvable?
\end{question}

We note that Theorem \ref{thm: lt5} shows that there is no nonabelian group $G$ with $\delta_{CD}(G) = 3$. 

\begin{question}
Is there a positive integer $m$ such that for every $k\geq m$ there is a nonabelian group $G$ with $\delta_{\CD}(G) = m$. 
\end{question}

%\section*{Acknowledgments.}
%The views expressed are those of the author and do not reflect the official policy or position of the \textsc{ARCYBER}, the Department of the Army, the Department of Defense, or the US Government. 

%The authors would like to thank Neelam Sangwan for helpful discussions regarding the paper. 
\bibliographystyle{amsalpha}
\bibliography{ref}

\end{document}

%% file: s3_alternative.tex
\section{Proof of Theorem \ref{thm: s3}.}\label{sec: s3}

If a finite group $G$ is not nilpotent, then it has at least four Sylow subgroups, and if $G$ has either four or five Sylow subgroups, then $G$ is a $\{2,3\}$-group. In this section we will sometimes use the non-standard notation $H_p$ to denote a Sylow $p$-subgroup of $G$. This makes our proofs easier to read by reminding the reader of the prime associated with the Sylow $p$-subgroup, especially as the prime $p$ will shift between certain lemmata. In addition, we wish to avoid confusion with $S_p$, the symmetric group on $p$ symbols. 

\begin{lemma}\label{lem: 1_2_4_3}
Suppose that a finite $\{2,3\}$-group $G$ has exactly four Sylow $3$-subgroups and a normal Sylow $2$-subgroup $H_2$. Then $\delta_{\CD}(G) > 5$.
\end{lemma}

\begin{proof}
Since $G$ has four Sylow $3$-subgroups, it is not nilpotent. By Theorem \ref{thm: nil}, we conclude that $\delta_{\CD}(G) \geq 5$.
The four Sylow $3$-subgroups of $G$ witness that $\delta_{\CD}(G) \geq 4$. By way of contradiction suppose that $\delta_{\CD}(G) = 5$.  

Note that $4$ divides $|H_2|$ which equals the index in $G$ of a Sylow $3$-subgroup.

Suppose $1\in \CD(G)$.  It follows from Lemma \ref{lem: no_p_groups} that all of the nontrivial $2$-groups and $3$-groups of $G$ are not in $\CD(G)$.  Since $\delta_{\CD}(G) = 5$, $G$ has exactly five total nontrivial $2$-groups or $3$-groups, namely the four Sylow $3$-groups and the normal Sylow $2$-group.  We conclude that all of the nontrivial Sylow subgroups have prime order which contradicts $|H_2| \geq 4$.

Now suppose that $1\notin \CD(G)$. By assumption, $1$ and the four Sylow $3$-subgroups are the only subgroups of $G$ not contained in $\CD(G)$. Hence $G$ and $H_2$ are in $\CD(G)$. This means that $m^*(G) = |G| \cdot |\textbf{Z}(G)|$ is divisible by $|G|$. However, $m^{*}(G) = |H_2|\cdot |\textbf{C}_G(H_2)|$. Since $H_2$ is a $2$-group, we conclude that $[H_2,H_3]=1$ for a Sylow $3$-subgroup $H_3$ of $G$. Then $H_2H_3 =G$ is a direct product and is thus nilpotent, a contradiction. Hence $\delta_{\CD}(G) > 5.$
\end{proof}

\begin{lemma}\label{lem:3_2_1_3}
Suppose that a finite $\{2,3\}$-group $G$ has exactly three Sylow $2$-subgroups and a normal Sylow $3$-subgroup $H_3$. If $\delta_{\CD}(G) = 5$, then $1\notin \CD(G)$.
\end{lemma}
\begin{proof}
Suppose by way of contradiction that $1\in \CD(G)$. It follows from Lemma \ref{lem: no_p_groups} that all of the nontrivial $2$-groups and $3$-groups of $G$ are not in $\CD(G)$.  Since $\delta_{\CD}(G) = 5$, $G$ has at most five nontrivial $2$-groups or $3$-groups, including the three Sylow $2$-subgroups and $H_3$.  And so $G$ has at most two subgroups that are nontrivial $3$-groups, and at most one subgroup that is a nontrivial $2$-group, but not Sylow.  This means that $G$ is a $\{2,3\}$-group with $|G| \leq 2^2 3^2$.  Considering each of the possible orders of $G$, $2 \cdot 3$, $4 \cdot 3$, $2 \cdot 9$, or $4 \cdot 9$, we see that in every case, the Sylow subgroups are abelian and one of them has Chermak--Delgado measure larger than $m^*(G) = |G|$.
\end{proof}

\begin{lemma}\label{lem: H_3inCD}
Suppose that a finite $\{2,3\}$-group $G$ has exactly three Sylow $2$-subgroups and a normal Sylow $3$-subgroup $H_3$. If $\delta_{\CD}(G) = 5$, then $H_3 \in \CD(G)$.\end{lemma}
\begin{proof}
By Lemma \ref{lem:3_2_1_3}, we know that $1\notin \CD(G)$. Hence $1$, together with the three Sylow $2$-subgroups witness that $\delta_{\CD}(G) \geq 4$. There is one other subgroup of $G$ not in $\CD(G)$. 

Suppose by way of contradiction that $H_3 \notin \CD(G)$. Let $H_2$ be a Sylow $2$-subgroup.  We claim that one of $H_3$ or $H_2$ must have prime order. Otherwise, if $1 < P < H_3$, and $1 < Q < H_2$, then $P,Q \in \CD(G)$, and so $1 = P \cap Q \in \CD(G)$, a contradiction. Let $K$ be a Sylow subgroup of prime order $p$. %Without loss of generality, suppose $T$ has prime order.  (The argument for $S$ would be similar.)

Then $K \cap \textbf{Z}(G) = 1$, as otherwise, $K \leq \textbf{Z}(G)$, and it would follow that the $\{2,3\}$-group $G$ is nilpotent. And so $\textbf{Z}(G)$ is a $p'$-group.  Now $m^*(G) = |G| \cdot |\textbf{Z}(G)| = p \cdot (p')^i$ for some $i$.  But also $K < K\textbf{Z}(G) < G$, and so $K\textbf{Z}(G) \in \CD(G)$.  But $K\textbf{Z}(G)$ is abelian, and so $K\textbf{Z}(G) \leq \textbf{C}_G(K\textbf{Z}(G))$, and so $p^2$ divides $|K\textbf{Z}(G)| \cdot |\textbf{C}_G(K\textbf{Z}(G))| = m^*(G)$, a contradiction.

\end{proof}

We can now prove Theorem \ref{thm: s3}, which states that for a non-nilpotent group if $\delta_{\CD}(G) = 5,$ then  $ G \cong S_3$. 

\begin{proof}[Proof of Theorem \ref{thm: s3}]
From Lemma \ref{lem: 1_2_4_3} we conclude that $G$ is a $\{2,3\}$ group with exactly three Sylow $2$-subgroups, and a normal Sylow $3$-subgroup $H_3$. From Lemmata \ref{lem:3_2_1_3} and \ref{lem: H_3inCD} we have that $H_3 \in \CD(G)$ and that the five subgroups of $G$ not in $\CD(G)$ are the three Sylow $2$-subgroups, the identity subgroup, and one other subgroup $X$. We will show that $X=G$. Suppose instead that $G\in \CD(G)$. 

Note that any Sylow $2$-subgroup of $\textbf{C}_G(H_3)$ is also a Sylow $2$-subgroup of $G$.  This is because $|H_3| \cdot |\textbf{C}_G(H_3)| = m^*(G)$ and since $G \in \CD(G)$, $|G|$ divides $m^*(G)$.  And so now we have that $G=H_3 H_2$ is a direct product where $H_2$ is a Sylow $2$-subgroup of $\textbf{C}_G(H_3)$, which implies that $G$ is nilpotent, a contradiction.

So $X=G$, and the three Sylow $2$-subgroups, the identity subgroup, and $G$ witness that $\delta_{\CD}(G) = 5$.

Let $H_2$ be a Sylow $2$-subgroup of $G$.  We argue that $|H_2|=2$.  Otherwise, there exists $1 < X \leq \textbf{Z}(H_2)$ with $X < H_2$, and so $X \in \CD(G)$. Then $m^*(G) = |X| \cdot |\textbf{C}_G(X)| = |H_3| \cdot |\textbf{C}_G(H_3)|$, and so $|H_3|$ divides $|\textbf{C}_G(X)|$. This implies that $H_3 \leq \textbf{C}_G(X)$. Hence $[X,H_2]=1$ and $[X,H_3]=1$. Since $\textbf{Z}(G) \leq H_3$, we have $X \leq \textbf{Z}(G) \leq H_3$, a contradiction. Hence $|H_2|=2$.  Since $|H_2|=2$, we note that $\textbf{C}_G(H_3)$ is a $3$-group as otherwise a Sylow $2$-subgroup would be contained in $\textbf{C}_G(H_3)$ which would imply that $G$ is nilpotent.  So $m^*(G) = |H_3| \cdot |\textbf{C}_G(H_3)|$ is a power of $3$.  

We now argue that $H_2$ is self-normalizing.  Otherwise, let $Y$ be a nontrivial Sylow $3$-subgroup of $\textbf{N}_G(H_2)$.  So $1 < Y < H_3$ (note that $Y < H_3$ as otherwise $H_3$ and $H_2$ would normalize one another, and so $G=H_3H_2$ would be nilpotent).  Then $H_2Y$ is a subgroup, and $H_2 < H_2Y < G$, and so $H_2Y$ is in $\CD(G)$, which contradicts the fact that $m^*(G)$ is a power of $3$.

So $H_2$ is self-normalizing, and $3 = |G:\textbf{N}_G(H_2)| = |G:H_2| = |H_3|$.  So $H_3$ is cyclic of order $3$, $H_2$ is cyclic of order $2$, and $G \cong S_3$.
\end{proof}

Figure \ref{fig: S3} shows the subgroup diagram of $S_3$. 

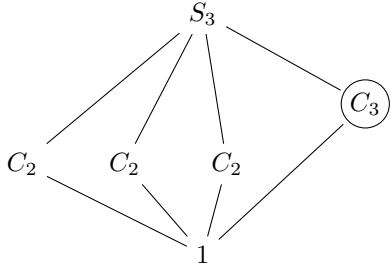
\begin{figure} 
\centering
\begin{tikzpicture}[scale=.8]
  \node (one) at (0,2) {$S_3$};
  \node (a) at (-3,-.5) {$C_2$};
  \node (b) at (-1.3,-.5) {$C_2$};
  \node (c) at (.4,-.5) {$C_2$};
  \node (d) at (2.7,.5) {$C_3$};
  \node (zero) at (0,-2) {$1$};
  \draw (zero) -- (a) -- (one) -- (b) -- (zero) -- (c) -- (one) -- (d) -- (zero);
  %\draw (2.35,.27)--(3.05,.27)--(3.05,.80)--(2.35,.80)--(2.35,.27);
  \draw (2.7,.5) circle (.4); 
\end{tikzpicture}
\caption{Subgroup diagram of $S_3$ with a circle around the Chermak--Delgado sublattice, which consists of the single subgroup isomorphic to $C_3$. This is the only non-nilpotent group with $\delta_{\CD}(G) = 5$.}  
\label{fig: S3}
\end{figure}

%% file: alternative_proofs.tex
\section{Proof of Theorem \ref{thm: lt5}.}\label{sec: lt5}

\begin{lemma}\label{lem: lt5->p-group}
Let $G$ be a nonabelian nilpotent group. If $\delta_{\CD}(G) < 7,$ then $G$ is a $p$-group. If $\delta_{\CD}{G} < 6$, then $G$ is a $2$-group or a $3$-group.
\end{lemma}

\begin{proof}
Suppose that $G$ is not a $p$-group, i.e., there are at least two prime divisors of $G$. Corollary \ref{cor: no_p_groups_nil} tells us that no $p$-groups are in $\CD(G)$. Moreover, since $G$ is nonabelian, then at least one Sylow subgroup $S$ of $G$ is nonabelian. Nonabelian $p$-groups have at least $p+2$ subgroups above the center, which means they have at least $p+4$ subgroups total. Thus $G$ contains $p+4$ distinct $p$-groups for the prime $p$.

Together with a nontrivial Sylow $q$-group $T$, these form a set of $p+5$ subgroups that witness $\delta_{\CD}(G)\geq p+5$. 

Thus we have \[p+5 \leq \delta_{\CD}(G) < 7,\] and we conclude that $G$ must be a $p$-group. 

Recall that a nonabelian $p$-group of odd order must contain $p+1$ subgroups of order $p$. By Lemma \ref{lem: order_p_witnesses} this means that 
\[p+1 \leq \delta_{\CD}(G) <6\] and we conclude that $p<4.$
 
\end{proof}

We note that for an odd prime $p$, $\delta_{\CD}(Q_8 \times C_p) = 7$ and is a nonabelian nilpotent group. We also note that the extraspecial group of order $5^3$ and exponent $25$ satisfies $\delta_{\CD}(G) = 6$, so both bounds in Lemma \ref{lem: lt5->p-group} are sharp.

\input{computation_conditions}

Lemma \ref{lem: lt5_conds} gives us conditions on the types of subgroups a $2$-group or $3$-group $G$ can have if $3\leq \delta_{\CD}(G) \leq 4$. This allows us to specify what types of subgroups $G$ can have of a given order. For example, a $2$-group $G$ with $3\leq \delta_{\CD}(G) \leq 4$ cannot contain the dihedral group of order $8$ as a subgroup because the dihedral group of order $8$ has 5 involutions. Similarly, it cannot contain the elementary abelian group of order $8$ as a subgroup. This means all of its subgroups of order $8$ are either cyclic, quaternion, or isomorphic to $C_4\times C_2$. We can computationally then search over all of the $2$-groups of a given order to look for potential subgroups of a 2-group $G$ with $3\leq \delta_{\CD}(G) \leq 4.$ 

We will pause to introduce the reader to another family of $p$-groups. These groups are commonly written as $M_{p^k} = \langle a, b | a^{p^{k-1}},b^{p}, a^b= a^{p^{k-2}+1}\rangle.$ For an odd prime $p$, the groups $M_{p^k}$ are the only nonabelian $p$-groups with cyclic maximal subgroups. 

\begin{lemma}\label{lem: 32}
Let $G$ be a $2$-group and suppose $3\leq \delta_{\CD}(G) \leq 4$. Then all subgroups of order $32$ in $G$ are isomorphic to one of the following groups: $C_{32},\, C_{16}\times C_2,\,$ or $ M_{32}$. 
\end{lemma}
\begin{proof}
Using \textsc{GAP} \cite{GAP} we can sort through all groups of a given order that meet the conditions of Lemma \ref{lem: lt5_conds}. For groups of order 32, we obtain the Venn diagram in Figure \ref{fig: 2-Venn}. Code to do this in the \textsc{GITHUB} repo: \href {https://github.com/7cocke/chermak\_delgado\_lt5}{https://github.com/7cocke/chermak\_delgado\_lt5}.  Different authors wrote the code independently of each other in the two languages. 

\begin{figure}
    \centering
    \includegraphics[scale = .7]{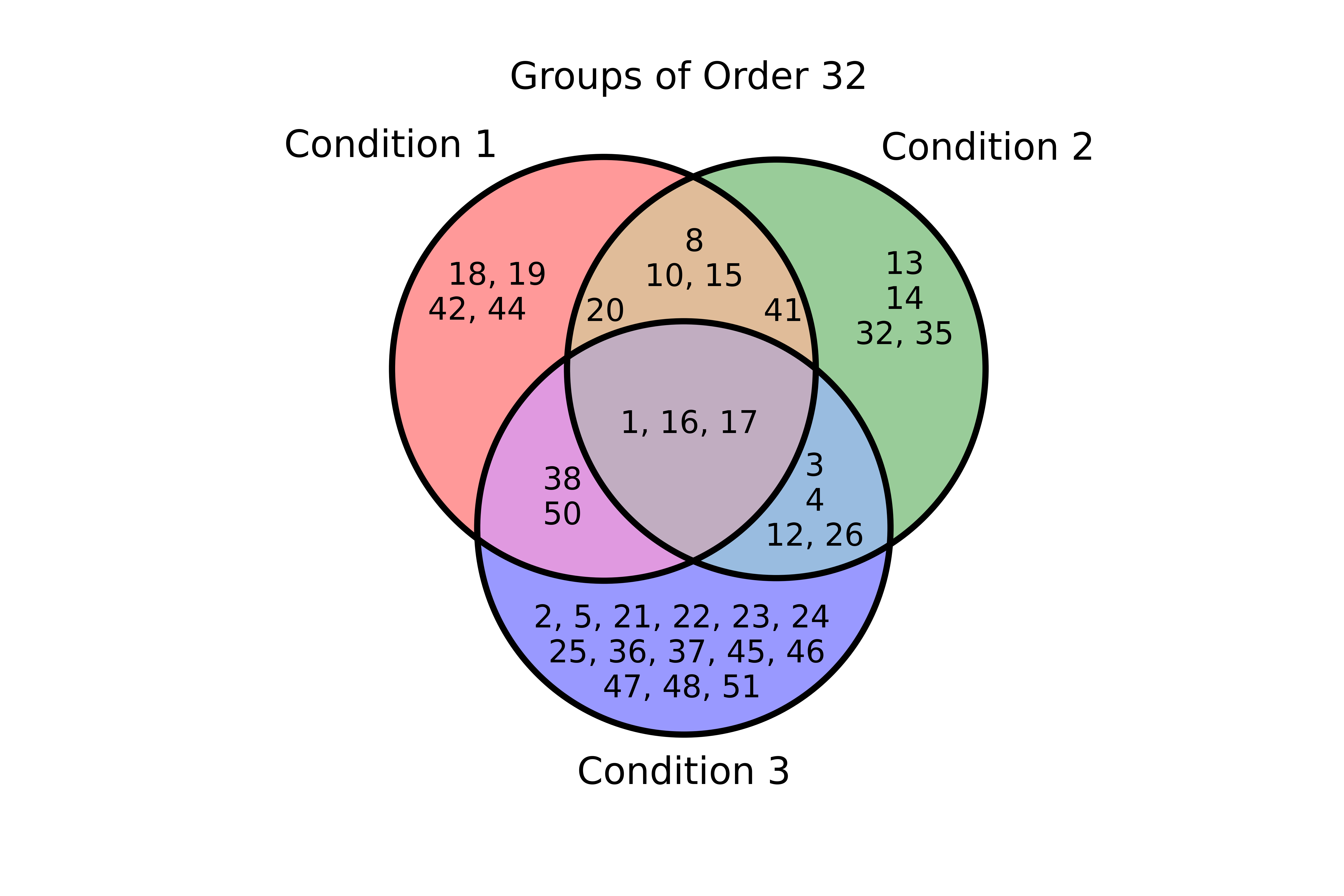}
    \caption{Groups of order 32 that satisfy at least one of conditions 1), 2), and 3) of Lemma \ref{lem: lt5_conds} identified by their SmallGroup index among groups of order 32, e.g., $1$ corresponds to SmallGroup(32,1) which is the cyclic group of order 32.}
    \label{fig: 2-Venn}
\end{figure}

\end{proof}

We could continue our calculations and would see that there are 3 possible subgroups of order 64 which are $C_{64}, C_{32} \times C_2$ and $M_{64}$. Similarly, we would see that there are 3 possible subgroups of order 128, i.e., the groups $C_{128}, C_{64}\times C_2, M_{128}.$ There are also 3 possible subgroups of order 256. This observation motivates the following theorem whose proof relies heavily on the exhaustive work of Berkovich and Janko \cite{BJ06}.

\begin{theorem}\label{thm: up}
Let $G$ be a $p$-group with order $|G| = p^{k+1}$ for $k>4$. Suppose that $G$ contains at most p+1 subgroups of order $p$ and that every maximal subgroup of $G$ is isomorphic to one of $C_{p^k}, C_{p^{k-1}}\times C_p$, and $M_{p^k}$. Then $G$ is isomorphic to one of the groups $C_{p^k+1}, C_{p^k} \times C_p$, or $M_{p^{k+1}}$.
\end{theorem}

\begin{proof}
If $G$ is abelian, then it must be isomorphic to either $C_{p^k}$ or $C_{p^{k-1}} \times C_p$. 

The rest of the proof will cite a number of results from the encyclopedic work of Berkovich and Janko \cite{BJ06} which classifies $p$-groups all of whose subgroups of index $p^2$ are abelian. 

Suppose that all maximal subgroups of $G$ are abelian. Then $G$ is minimal nonabelian. Berkovich and Janko \cite[\textbf{3.1}]{BJ06} attribute the classification of minimal nonabelian $p$-groups to R\'edei. Our group $G$ contains at most $p+1$ subgroups of order $p$ and has order greater than $p^4$. If $G$ were minimal nonabelian, it would have to be isomorphic to $M_{p^{k+1}}$. 

Hence $G$ must contain a nonabelian maximal subgroup. Thus $G$ is an $A_2$-group, i.e., a group all of whose subgroups of index $p^2$ are abelian. Moreover, all proper subgroups of $G$ would be metacyclic. If $G$ itself were not metacyclic, then it would be minimal nonmetacyclic. Berkovich and Janko \cite[\textbf{1.1(l)}]{BJ06} cite Blackburn \cite[\textbf{3.2}]{Blackburn61} for the classification of minimal nonmetacyclic groups, none of which could be our group $G$. Hence the group $G$ is a metacyclic $A_2$-group. Such groups are classified by Berkovich and Janko \cite[\textbf{5.2}]{BJ06}. Again, since we have restricted the maximal subgroups of $G$, it cannot be any of these groups. 

We conclude that $G$ is either abelian or isomorphic to $M_{p^{k+1}}$. 
\end{proof}

For $3$-groups, Lemma \ref{lem: lt5_conds} can be used to establish the following.

\begin{lemma}
Let $G$ be a nonabelian $3$-group with  $\delta_{\CD}(G) = 4$. Then all subgroups of order $243$ in $G$ are isomorphic to one of $C_{243},\, C_{81}\times C_3,\,$ or $M_{243}$.
\end{lemma}

These are the same 3 types of groups that occur in Theorem \ref{thm: up}. This means that as in the 2-group case, for a $3$-group with $\delta_{\CD}(G) = 4,$ there are at most three isomorphism classes of subgroups for every order greater than 243.

\begin{theorem} \label{thm: lt52}
If $G$ is a nonabelian $2$-group and $\delta_{\CD}(G) \leq 4$, then $G$ is the quaternion group of order $8$.
\end{theorem}

\begin{proof}
If $G$ has a single involution, then $G$ is generalized quaternion and we note that $m^*(G) = \left(\frac{|G|}{2}\right)^2$. If $|G| > 8$, then $\CD(G)$ is a single subgroup and $\delta_{\CD}(G) > 4.$

Suppose by way of contradiction that $G$ is not the quaternion group of order $8$. We must have that $3\leq \delta_{\CD}(G) \leq 4$. As noted, in Lemma \ref{lem: 32} the only subgroups of order $32$ in $G$ would be isomorphic to $C_{32}, C_{16}\times C_2$, or $M_{32}$. By Theorem \ref{thm: up}, the only subgroups of order 64 in $G$ would be isomorphic to $C_{64}, C_{32}\times C_2$ or $M_{64}$. Continuing in this manner, we see that $G$ itself would be isomorphic to either $C_{2^k}, C_{2^{k-1}}\times C_{2}$ or $M_{2^k}$. However, none of these groups satisfy $\delta_{\CD}(G) \leq 4$. 

Thus $|G| \leq 32$. A computational search using either $\textsc{GAP}$ or $\textsc{MAGMA}$ shows that there is no such nonabelian $2$-group. 
\end{proof}

We note that the dihedral group of order $8$ has exactly $5$ groups not contained in $\CD(G)$, i.e., $\delta_{\CD}(G) = 5$.

\begin{theorem}\label{thm: lt53}
If $G$ is a nonabelian $3$-group, then $\delta_{\CD}(G) \geq 4$ with equality if and only if $G$ is the extraspecial group of order 27 and exponent $9$. 
\end{theorem}

\begin{proof}
Using the same code as in Lemma \ref{lem: 32} we see that all of the subgroups of $G$ of order 243 are isomorphic to $C_{243}, C_{81} \times C_3$, or $M_{243}$. By Theorem $\ref{thm: up}$, either $|G|\leq 81$ or $G$ itself is isomorphic to one of $C_{3^k}, C_{3^{k-1}}\times C_3,$ or $M_{3^k}$. Regardless, $G$ does not satisfy $\delta_{\CD}(G) = 4$ except for when $G$ is $M_{27}$ which is the extraspecial group of order 27 and exponent 9. 
\end{proof}

In Figure \ref{fig: M27} we display the subgroup diagram of $M_{27}$ the only nonabelian nilpotent group with $\delta_{CD}(G) =4$.

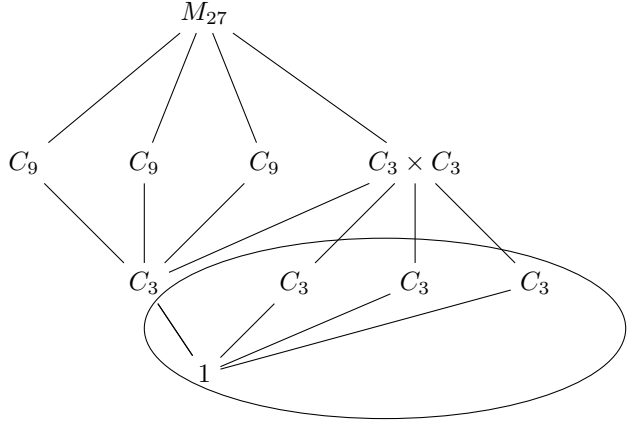
\begin{figure}
\centering
\begin{tikzpicture}[scale=.8]
  \node (one) at (0,4) {$M_{27}$};
  \node (a) at (-3,1.5) {$C_9$};
  \node (b) at (-1,1.5) {$C_9$};
  \node (c) at (1,1.5) {$C_9$};
  \node (d) at (3.5,1.5) {$C_3\times C_3$};
  \node (e) at (3.5,-.5) {$C_3$};
  \node (f) at (1.5,-.5) {$C_3$};
  \node (g) at (5.5,-.5) {$C_3$};
  \node (h) at (-1,-.5) {$C_3$};
  \node (zero) at (0,-2) {$1$};
  \draw (a) -- (one) -- (b) -- (h) -- (c) -- (one) -- (d);
  \draw (a) -- (h) -- (zero);
  \draw (d) -- (e) -- (zero) -- (f) -- (d) -- (h) -- (zero) -- (g) -- (d);
  \draw (3,-1.25) ellipse (4cm and 1.5cm);
  %\draw (-.5,-2.2)--(.35,.04)--(6.05,.04)--(5.25,-2.2)--(-.5,-2.2);
\end{tikzpicture}
\caption{The subgroup diagram of $G=M_{27}$, i.e., the extraspecial group of order 27 and exponent 9. The four subgroups not in $\CD(G)$ are contained wthin the ellipse in the figure. This is the only nonabelian group with $\delta_{CD}(G) = 4$.}
\label{fig: M27}
\end{figure}

Combining Theorems \ref{thm: lt52} and \ref{thm: lt53} provides the proof of Theorem \ref{thm: lt5}.

\begin{proof}[Proof of Theorem \ref{thm: lt5}]
If $G$ is nonabelian, then $G$ is either a $2$-group or a $3$-group by Lemma \ref{lem: lt5->p-group}. Combining Theorem \ref{thm: lt52} and \ref{thm: lt53}, we see that $G$ must be the extraspeical group of order $27$ and exponent $9$. 

If $G$ is abelian, then we know that $\delta_{\CD}(G)$ is equal to the number of subgroups of $G$ minus 1. Counting the subgroups of abelian groups finishes the proof.
\end{proof}

%% file: computation_conditions.tex
\begin{lemma}\label{lem: no_quaternion}
Let $G$ be a nonabelian $p$-group and suppose that $3 \leq \delta_{\CD}(G) \leq 4$. Then $G$ contains at least $p+1$ subgroups of order $p$.  Of these, $p$ of the subgroups together with the identity witness that $\delta_{CD}(G) \geq p+1$.
\end{lemma}
\begin{proof}
If $p>2$, then every nonabelian $p$-group contains at least $p+1$ subgroups of order $p$. 

If $p=2$, and $G$ contains a single subgroup of order $2$, then $G$ is generalized quaternion. Since the generalized quaternion 2-groups of order greater than $8$ contain more than 7 subgroups and contain a single subgroup in their Chermak--Delgado lattices, we conclude that for such groups $\delta_{\CD}(G) > 4$. Moreover $\delta_{\CD}(Q_8) = 1$. 
\end{proof}

Hence for a nonabelian $p$-group $G$ with $3\leq \delta_{\CD}(G) \leq 4$ there are $p$ subgroups of order $p$, which together with the identity witness that $\delta_{CD}(G) \geq p+1$. We will use this to derive a set of conditions for a subgroup of $G$ to satisfy. 

\begin{lemma}\label{lem: lt5_conds}
Let $G$ be a nonabelian $p$-group and suppose $3\leq \delta_{\CD}(G) \leq 4$. For every $K\leq G$ the following hold:
\begin{enumerate}
    \item[1)] There are at most $4$ subgroups of order $p$ in $K$.
    \item[2)] There is a subgroup $Z\leq \mathbf{Z}(K)$ with $|Z|=p$ such that \[|\{H\leq K: p^2\leq |H|\text{ and } Z\not < H\}| \leq \begin{cases} 0 & p > 2 \\ 1 & p = 2.\end{cases}\]
    %\item[3)] \begin{itemize}\item[(a)] If $p>2$, then all subgroups $H\leq K$ satisfy $[[Z(H):H]] = \CD(H)$.
    %\item[(b)] If $p=2$, then no more than one subgroup $H\leq K$ can satisfy $\left|[[\textbf{Z}(H):H]]\setminus \CD(H) \right| \geq 2$.
    %\end{itemize}
    \item[3)] If $|K|\geq p^3$, then $K \in \CD(K)$ and at most one subgroup $H$ of $K$ can satisfy $\textbf{Z}(K) \leq H$ and $H\notin \CD(K)$. 
\end{enumerate}
\end{lemma}

\begin{proof}
As noted in Lemma \ref{lem: no_quaternion}, the group $G$ has at least $p+1$ subgroups not in $\CD(G)$ consisting of the identity and the $p$ subgroups of order $p$. 
1) If $K$ had more than 4 subgroups of order $p$, then it would have at least $5$ such groups. Hence $G$ would have at least 5 subgroups of order $p$ and by Lemma \ref{lem: order_p_witnesses} we would have that $\delta_{\CD}(G) \geq 5$. 

2)  If $K$ has a single subgroup of order $p$ it is cyclic or generalized quaternion and all groups of order greater than $p^2$ contain the unique subgroup of order $p$. 
Now suppose that $K$ has $p+1$ involutions. Hence $G$ has $p+1$ witnesses from Lemma \ref{lem: no_quaternion}. There can only be one additional witness with order greater than $p$ if $p=2$. 

Suppose that $p=2$. Then $K$ contains the three subgroups $A=\langle a \rangle $, $B=\langle b\rangle$, and $C=\langle ab\rangle$ of order 2. If none of them are in $\CD(G)$, then the four subgroups not in $\CD(G)$ are $\{1, A,B,C\}$.  Now $\langle a, b\rangle$ contains a central involution, and hence is a product of any two of the subgroups of order $2$ and has order $4$.  Now $\langle a, b\rangle$ is in $\CD(G)$, and it follows that $\textbf{Z}(G) = \langle a, b\rangle$. Moreover, since $\delta_{\CD}(G) = 4$ in this case, we conclude that every subgroup of $G$ of order $4$ or larger is in $\CD(G)$, and thus contains $\textbf{Z}(G) = \langle a,b\rangle$.

Otherwise let $Z\leq K$ of order 2 in $\CD(G)$. Since $Z\in \CD(G)$, we know that $Z=\mathbf{Z}(G)$ and hence $Z\leq \mathbf{Z}(K)$. Moreover, since the other two involutions and the identity would not be in $\CD(G)$, we have $\delta_{\CD}(G) \geq 3$ and thus at most one other subgroup of $G$ (which must have order greater than 2) could not be in $\CD(G)$ and thus not contain $Z$. 

Suppose that $p>2$. Then by Lemma \ref{lem: lt5->p-group}, $p=3$ and from \ref{lem: no_quaternion} we note that exactly one subgroup of order $3$ is in $\CD(G)$; call this subgroup $H$ and note that $H=\textbf{Z}(G)$. Then every subgroup of $G$ of order greater than $3$ is in $\CD(G)$ and must contain $H$ as a subgroup. 

%3. (a) If $p=3$, then $G$ contains exactly $4$ subgroups of order $3$, one of which is $\textbf{Z}(G)$ and contained in $\textbf{Z}(H)$ for all $H$ of order greater than or equal to $p^2= 9.$  Suppose that $H\leq K$ and $[[Z(H):H]] \neq \CD(H)$. If $H\in \CD(G)$, then $[[Z(H):H]] \cap \CD(G) = \CD(H)$, and we conclude that there is a subgroup in $[[Z(H):H]]$ not in $\CD(G)$. But, this subgroup has order greater than $3$, and together with the 3 noncentral subgroups of order 3  and the identity would witness that $\delta_{\CD}(G) \geq 5$. But, if $H\notin \CD(G)$, then $H$ together with the 3 noncentral subgroups and the identity would witness that $\delta_{\CD}(G) \geq 5$. 

%\sloppy (b) Suppose $H,J \leq K$ with $H \neq J$ such that $\left|[[\textbf{Z}(H):H]]\setminus \CD(H) \right| \geq 2$ and $\left|[[\textbf{Z}(J):J]]\setminus \CD(J) \right| \geq 2$. If both $H,J$ are not in $\CD(G)$, then $\delta_{\CD}(G) \geq 5$. Without loss of generality, suppose that $H\in \CD(G)$. Then there exist two subgroups in $[[\textbf{Z}(H):H]]\setminus \CD(H)$  that are not in $\CD(G)$ and we conclude that $\delta_{\CD}(G) \geq 5$.

3) We note that if $K$ is cyclic then $\mathbf{Z}(K)=K \in \CD(K)$. Suppose that $|K| \geq p^3$ is not cyclic and by way of contradiction that $K\not \in \CD(K)$. By Lemma \ref{lem: interval} we have that $K\not \in \CD(G)$. Moreover, since $\CD(G)$ is closed under products we must have that at least $p$ of the maximal subgroups of $K$ (which have order $\geq p^2$) are not in $\CD(G)$. Hence $\delta_{\CD}(G) \geq 2(p+1) > 5$. We conclude that $K\in \CD(K)$.

 Suppose that $H$ and $J$ are subgroups of $K$ with $\mathbf{Z}(K)\leq H$ and $\mathbf{Z}(K)\leq J$ such that neither $H$ nor $J$ are in $\CD(K)$. Then $H$ and $J$ are not in $\CD(G)$ by Lemma \ref{lem: interval}. Since $K\in \CD(K)$, we conclude that $\mathbf{Z}(K) < H$ and $\mathbf{Z}(K) < J$ and thus the orders of both $H$ and $J$ are greater than or equal to $p^2$. Hence, $H$, $J$, the identity, and $p$ of the subgroups of order $p$ in $G$ would witness that $\delta_{\CD}(G) \geq p+3 \geq 5$. 
\end{proof}